\documentclass[letterpaper]{article}

\usepackage{amsmath}
\usepackage{amsfonts}
\usepackage{amsthm}
\usepackage[letterpaper]{geometry}
\usepackage{hyperref}
\usepackage[]{ytableau}
\usepackage[all, arc, curve, frame]{xy}

\numberwithin{equation}{section}

\title{Armstrong's Conjecture for $(k, mk + 1)$-Core Partitions}
\author{Amol Aggarwal}

\begin{document}

\maketitle

\begin{abstract}
A conjecture of Armstrong states that if $\gcd (a, b) = 1$, then the average size of an $(a, b)$-core partition is $(a - 1)(b - 1)(a + b + 1) / 24$. Recently, Stanley and Zanello used a recursive argument to verify this conjecture when $a = b - 1$. In this paper we use a variant of their method to establish Armstrong's conjecture in the more general setting where $a$ divides $b - 1$. 
\end{abstract}

\newtheorem{thm}{Theorem}[section]
\newtheorem{prop}[thm]{Proposition}
\newtheorem{lem}[thm]{Lemma}
\newtheorem{cor}[thm]{Corollary}
\newtheorem{conj}[thm]{Conjecture}

\section{Introduction}

A {\itshape partition} is a finite, nonincreasing sequence $\lambda = (\lambda_1, \lambda_2, \ldots , \lambda_r)$ of positive integers. The sum $\sum_{i = 1}^r \lambda_i$ is the {\itshape size} of $\lambda$ and is denoted by $|\lambda|$. We may represent $\lambda$ by a {\itshape Young diagram}, which is a collection of $r$ left-justified rows of cells with $\lambda_i$ cells in row $i$. The {\itshape hook length} of any cell $C$ in the Young diagram is defined to be the number of cells to the right of, below, or equal to $C$. For instance, Figure 1 shows the Young diagram and hook lengths of the partition $(5, 3, 1, 1)$. 

For any positive integers $a$ and $b$, a partition is called an {\itshape $(a, b)$-core} if no cell in its Young diagram has hook length equal to $a$ or $b$; for instance, Figure 1 shows that $(5, 3, 1, 1)$ is a $(3, 7)$-core. Simultaneous core partitions have been the topic of many articles during the past decade (see \cite{3, 4, 5, 7, 8, 9, 10, 11, 14, 15}). They are particularly interesting when $\gcd (a, b) = 1$. In this case, there are only finitely many $(a, b)$-cores; in fact, a theorem of Anderson states that there are $\binom{a + b}{a} / (a + b)$ such cores \cite{4}. 

The proof of Anderson's theorem is through a bijective correspondence between $(a, b)$-cores and order ideals of the poset $P_{a, b}$, whose elements are all positive integers not contained in the numerical semigroup generated by $\{ a, b \}$ and whose partial order is fixed by requiring $p \in P_{a, b}$ to cover $q \in P_{a, b}$ if $p - q$ is either $a$ or $b$ (throughout the article, we will follow the poset terminology given in Chapter 3 of Stanley's text \cite{12, 13}). Specifically, this correspondence sends an $(a, b)$-core partition $\lambda = (\lambda_1, \lambda_2, \ldots , \lambda_r)$ to the order ideal $I_{\lambda} = \{ \lambda_1 + r - 1, \lambda_2 + r - 2, \ldots , \lambda_r \} \in J (P_{a, b})$, where $J(P)$ denotes the set of order ideals of any poset $P$; observe that $I_{\lambda}$ consists of the hook lengths in the leftmost column of the Young diagram of $\lambda$. From this bijection, we deduce the identity 
\begin{flalign}
\label{partitionideals}
|\lambda| = \sigma (I_{\lambda}) - \binom{|I_{\lambda}|}{2}, 
\end{flalign}

\noindent where $\sigma (I_{\lambda}) = \sum_{i \in I_{\lambda}} i$ is the sum of the elements in $I_{\lambda}$. 

To see an example of this bijection, let $(a, b) = (3, 7)$; then, $P_{3, 7} = \{ 1, 2, 4, 5, 8, 11 \}$. In this poset, $11$ covers $4$ and $8$; $8$ covers $1$ and $5$; $5$ covers $2$; and $4$ covers $1$. The $(3, 7)$-core $(5, 3, 1, 1)$ corresponds to the order ideal $ \{ 8, 5, 2, 1 \} \subset P_{3, 7}$; equation \hyperref[partitionideals]{(\ref*{partitionideals})} may be verified since $(5, 3, 1, 1)$ has size $10$ and $\sigma (\{ 8, 5, 2, 1 \}) = 16$. 

\begin{figure}
\begin{center}
\ytableausetup{centertableaux}
\begin{ytableau}
8 & 5 & 4 & 2 & 1 \\
5 & 2 & 1 \\
2 \\
1
\end{ytableau}

\caption{The Young diagram of (5, 3, 1, 1) is shown above; each cell contains its hook length. } 
\end{center}
\end{figure}

In 2011, Armstrong informally proposed the following conjecture that predicts the average size of an $(a, b)$-core; this conjecture was later published in \cite{5}.  
\begin{conj}
\label{conjecture}
If $\gcd (a, b) = 1$, then 
\begin{flalign}
\label{sum}
\displaystyle\sum_{\lambda} |\lambda| = \displaystyle\frac{(a - 1)(b - 1)(a + b + 1)}{24 (a + b)} \binom{a + b}{a}, 
\end{flalign}

\noindent where $\lambda$ is summed over all $(a, b)$-cores. Equivalently, the average size of an $(a, b)$-core is $(a - 1)(b - 1)(a + b + 1) / 24$. 
\end{conj}

In addition to having an intrinsic appeal, a proof of \hyperref[conjecture]{Conjecture \ref*{conjecture}} would yield implications about numerical semigroups generated by two elements. Yet, despite the ostensible simplicity of \hyperref[sum]{(\ref*{sum})}, it remains unproven. However, there have recently been several partial results towards Armstrong's conjecture. In 2013, Stanley and Zanello used a recursive method to prove \hyperref[conjecture]{Conjecture \ref*{conjecture}} when $a = b - 1$ \cite{14}. In response to another conjecture in \cite{4}, Chen, Huang, and Wang later established an analog of Armstrong's conjecture for self-conjugate core partitions using the Ford-Mai-Sze bijection \cite{7}. 

In this paper we use a variant of the recursive method given by Stanley and Zanello to verify a more general case of \hyperref[conjecture]{Conjecture \ref*{conjecture}}. In particular, we prove the theorem below. 

\begin{thm}
\label{partitiontheoremgeneral}
For any integers $k, m \ge 1$, Armstrong's conjecture holds for $(k, mk + 1)$-cores. Specifically, 
\begin{flalign*}
\displaystyle\sum_{\lambda} |\lambda| = \displaystyle\frac{mk (k - 1) ((m + 1)k + 2)}{24 (mk + 1)} \binom{(m + 1)k}{k}, 
\end{flalign*}

\noindent where $\lambda$ is summed over all $(k, mk + 1)$-cores. 
\end{thm}

\noindent As Stanley and Zanello did in the case $m = 1$, we will prove the theorem above by using Anderson's bijection and the manageable behavior of the poset $P_{k, mk + 1}$. Applying \hyperref[partitionideals]{(\ref*{partitionideals})}, we see that \hyperref[partitiontheoremgeneral]{Theorem \ref*{partitiontheoremgeneral}} is equivalent to the theorem below.

\begin{thm}
\label{idealtheoremgeneral}
For any integers $m, k \ge 1$, 
\begin{flalign}
\label{sumidealsgeneral}
\displaystyle\sum_{I \in J (P_{k, mk + 1})} \left( \sigma(I) - \binom{|I|}{2} \right) = \displaystyle\frac{mk (k - 1) ((m + 1)k + 2)}{24 (mk + 1)} \binom{(m + 1)k}{k}. 
\end{flalign}
\end{thm}

\noindent We will prove this theorem in Section 3.

\section{Proof of Theorem 1.3 When $m = 2$}
The proof of \hyperref[idealtheoremgeneral]{Theorem \ref*{idealtheoremgeneral}} is quite computational, so we will first verify \hyperref[idealtheoremgeneral]{Theorem \ref*{idealtheoremgeneral}} when $m = 2$. In particular, we will establish the following result. 

\begin{thm}
\label{idealtheorem2}
For any integer $k \ge 1$, 
\begin{flalign}
\label{sumideals2}
\displaystyle\sum_{I \in J(P_{k, 2k + 1})} \left( \sigma(I) - \binom{|I|}{2} \right) = \displaystyle\frac{k (k - 1) (3k + 2)}{12 (2k + 1)} \binom{3k}{k}. 
\end{flalign}
\end{thm}

Let us begin by introducing some notation. For each integer $n$, let $P_n = P_{n, 2n + 1}$. Let $Q_n$ be the poset obtained by removing the minimal elements of $P_{n + 1}$; equivalently, $Q_n = P_{n + 1} \backslash \{ 1, 2, \ldots , n \}$. Figure 2 depicts the Hasse diagrams of $P_4$ and $Q_3$. Let $A_n$ denote the number of order ideals in $P_n$ and let $B_n$ denote the number of order ideals in $Q_n$. By a theorem of Bizley (see \cite{6}), 
\begin{flalign}
\label{expressiona}
A_n = \binom{3n + 1}{n} / (3n + 1).
\end{flalign} Furthermore, one may check that there is a poset isomorphism $Q_n \simeq P_{n + 1, 2n + 1}$ under the map sending $q \in Q_n$ to $(2n + 2) \lfloor q / (n + 1) \rfloor - q \in P_{n + 1, 2n + 1}$; applying the theorem of Bizley again yields that $B_n = \binom{3n + 2}{n + 1} / (3n + 2)$. 

Define the generating functions $A(x) = \sum_{k = 0}^{\infty} A_k x^k$ and $B(x) = \sum_{k = 0}^{\infty} B_k x^k$, where $x$ is a formal variable. It is known (see \cite{1, 2}) that these generating functions have explicit forms given by 
\begin{flalign*}
A(x) = \displaystyle\frac{2}{\sqrt{3x}} \sin \left( \frac{\arcsin (\sqrt{27x / 4})}{3} \right)
\end{flalign*}

\noindent and 
\begin{flalign}
\label{ab}
B(x) &= \displaystyle\frac{4}{3x} \sin^2 \left(\frac{\arcsin (\sqrt{27x / 4})}{3} \right) \nonumber \\
&= A(x)^2.  
\end{flalign}

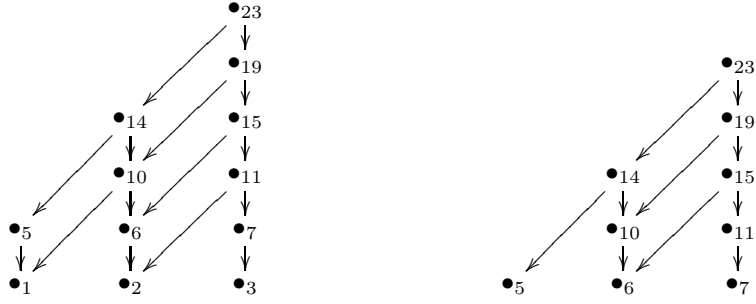
\begin{figure}[t]
\begin{center}
\xymatrix @R=.75pc {& & & & \bullet_{23} \ar[d] \ar[ldd] \\
& & & & \bullet_{19} \ar[d] \ar[ldd] & & & & & \bullet_{23} \ar[d] \ar[ldd] \\
& & & \bullet_{14} \ar[d] \ar[ldd] & \bullet_{15} \ar[d] \ar[ldd] & & & & & \bullet_{19} \ar[d] \ar[ldd] \\
& & & \bullet_{10} \ar[d] \ar[ldd] & \bullet_{11} \ar[d] \ar[ldd] & & & & \bullet_{14} \ar[d] \ar[ldd] & \bullet_{15} \ar[d] \ar[ldd] \\
& & \bullet_5 \ar[d] & \bullet_6 \ar[d] & \bullet_7 \ar[d] & & & & \bullet_{10} \ar[d] & \bullet_{11} \ar[d] \\
& & \bullet_1 & \bullet_2 & \bullet_3 & & & \bullet_5 & \bullet_6 & \bullet_7 }
\caption{The Hasse diagrams of the posets $P_4$ and $Q_3$ are shown to the left and right, respectively.}
\end{center} 
\end{figure} 

\noindent From \cite{1} (or from \hyperref[ab]{(\ref*{ab})}), we have that 
\begin{flalign}
\label{relation2}
x A(x)^3 - A(x) + 1 = 0. 
\end{flalign}

\noindent For each $p \in P_n$, let $\rho_{P_n} (p) = \lfloor p / n \rfloor$; for each $q \in Q_n$, let $\rho_{Q_n} (q) = \rho_{P_{n + 1}} (q)$. For each $S \in \{ P, Q \}$, define the sums
\begin{flalign*}
T_n^{(S)} = \displaystyle\sum_{I \in J (S_n)} |I|; \quad R_n^{(S)} = \displaystyle\sum_{I \in J (S_n)} \displaystyle\sum_{i \in I} \rho_{S_n} (i); \quad G_n^{(S)} = \displaystyle\sum_{I \in J (S_n)} \left( \sigma(I) - \binom{|I|}{2} \right). 
\end{flalign*}

\noindent Also define the generating functions

\begin{flalign*}
T_S (x) = \displaystyle\sum_{k = 0}^{\infty} T_k^{(S)} x^k; \quad R_S (x) = \displaystyle\sum_{k = 0}^{\infty} R_k^{(S)} x^k; \quad G_S (x) = \displaystyle\sum_{k = 0}^{\infty} G_k^{(S)} x^k. 
\end{flalign*}

\noindent The equality \hyperref[sumideals2]{(\ref*{sumideals2})} is equivalent to 
\begin{flalign*}
G_n^{(P)} = \displaystyle\frac{n(n - 1)(3n + 2)}{12 (2n + 1)} \binom{3n}{n}. 
\end{flalign*}

\noindent Therefore, by \hyperref[expressiona]{(\ref*{expressiona})}, it suffices to show that 
\begin{flalign}
\label{function2}
G_P (x) = \displaystyle\frac{3 x^3 A''' (x) + 8 x^2 A'' (x)}{12}
\end{flalign}

\noindent in order to establish \hyperref[idealtheorem2]{Theorem \ref*{idealtheorem2}}. 

In order to prove \hyperref[function2]{(\ref*{function2})}, we will derive several recursions that yield algebraic relations between the generating functions $A$, $T$, $R$, and $G$. These relations will allow us to solve for $G$ as a rational function in $x$ and $A$ and thereby deduce the above equality. 

To obtain these recursions, we partition the sets $J (P_n)$ and $J (Q_n)$ in a way similar to that done by Stanley and Zanello in \cite{14}. For each integer $i \in [1, n]$, let $J_i (P_n) \subset J (P_n)$ denote the set of order ideals of $P_n$ that contain $\{ 1, 2, \ldots , i - 1 \}$ but not $i$. Similarly, for each integer $i \in [1, n + 1]$, let $J_i (Q_n)$ denote the set of order ideals of $Q_n$ that contain $\{ n + 2, n + 3, \ldots , n + i \}$ but not $n + i + 1$. One may refer to Figure 3 for examples. On the left is the Hasse diagram of $P_6$; any ideal in $J_3 (P_6) $ must contain the elements labelled by squares, must avoid the elements labelled by white circles, and may contain some of the elements labelled by black circles; the analogous figure for $J_4 (Q_5)$ is shown on the right. We have the decompositions $J (P_n) = \bigcup_{i = 1}^n J_i (P_n)$ and $J (Q_n) = \bigcup_{i = 1}^{n + 1} J_i (Q_n)$. 

For any integer $i \in [1, n]$, let $P_n (i) \subset P_n$ denote the poset of elements in $P_n$ that are greater than some integer in $[1, i - 1]$ and incomparable to each integer in $[i, n - 1]$, where the ordering is with respect to the poset $P_n$; for instance, $P_6 (3) = \{ 7, 8, 14, 20 \}$, as seen from Figure 3. Also let $P_n (-i)$ denote the poset of elements $p \in P_n$ that are incomparable to each integer in $[1, i]$; for instance, $P_6 (-3) = \{ 4, 5, 10, 11, 17, 23 \}$. Observe that there is a poset isomorphism $Q_{i - 1} \simeq P_n (i)$ under the map sending an $q \in Q_{i - 1}$ to $q + (n - i) \rho_{Q_{i - 1}} (q) \in P_n (i)$. Similarly, there is a poset isomorphism $P_{n - i} \simeq P_n (-i)$ that maps $p \in P_{n - i}$ to $p + i + i \rho_{P_{n - i}} (p) \in P_n (-i)$. 

For any integer $i \in [1, n]$, let $Q_n (i) \subset Q_n$ denote the poset of elements in $Q_n$ that are greater than some integer in $[n + 2, n + i]$ and incomparable to each integer in $[n + i + 1, 2n + 1]$, where the ordering is with respect to the poset $Q_n$; for instance, $Q_5 (4) = \{ 14, 15, 20, 21, 27, 33 \}$, as seen from Figure 3. Also let $Q_n (-i)$ denote the poset of elements in $Q_n$ that are incomparable to each integer in $[n + 2, n + i + 1]$; for instance, $Q_5 (-4) = \{ 11, 17 \}$. Observe that there is a poset isomorphism $P_{i - 1} \simeq Q_n (i)$ under the map that sends $p \in P_{i - 1}$ to $p + 2n + 3 + (n + 2 - i) \rho_{P_{i - 1}} (p) \in Q_n (i)$. Similarly, there is a poset isomorphism $P_{n - i + 1} \simeq Q_n (-i)$ that maps $p \in P_{n - i + 1}$ to $p + n + 1 + i + i \rho_{P_{n - i + 1}} (p) \in Q_n (-i)$. 

We will now deduce the following recursive identities. 

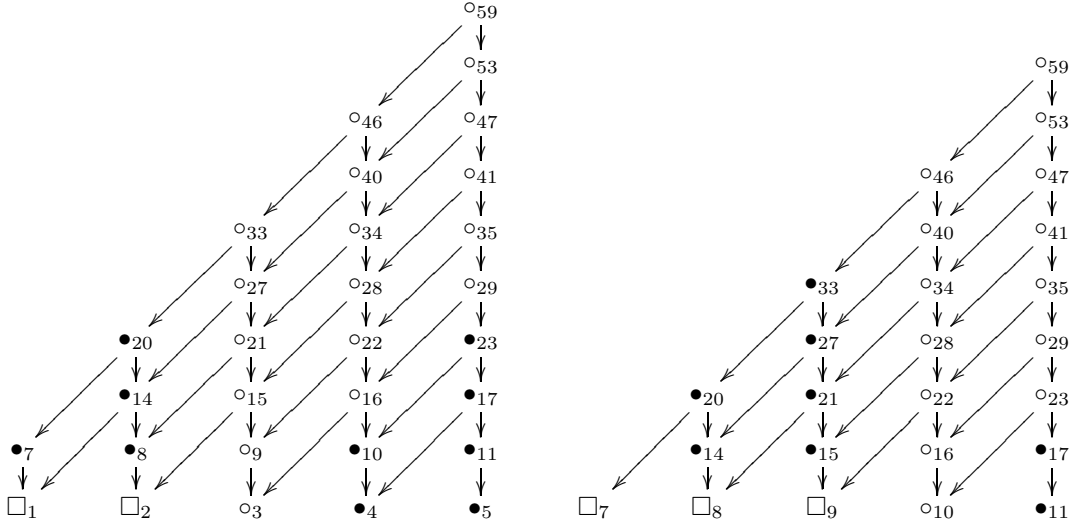
\begin{figure}[t]
\begin{center}
\xymatrix @R=.75pc {& & & & \circ_{59} \ar[d] \ar[ldd] \\ 
& & & & \circ_{53} \ar[d] \ar[ldd] & & & & & \circ_{59} \ar[d] \ar[ldd]  \\ 
& & & \circ_{46} \ar[d] \ar[ldd] & \circ_{47} \ar[d] \ar[ldd] & & & & & \circ_{53} \ar[d] \ar[ldd] \\
& & & \circ_{40} \ar[d] \ar[ldd] & \circ_{41} \ar[d] \ar[ldd] & & & & \circ_{46} \ar[d] \ar[ldd] & \circ_{47} \ar[d] \ar[ldd] \\ 
& & \circ_{33} \ar[d] \ar[ldd] & \circ_{34} \ar[d] \ar[ldd] & \circ_{35} \ar[d] \ar[ldd] & & & & \circ_{40} \ar[d] \ar[ldd] & \circ_{41} \ar[d] \ar[ldd] \\
& & \circ_{27} \ar[d] \ar[ldd] & \circ_{28} \ar[d] \ar[ldd] & \circ_{29} \ar[d] \ar[ldd] & & & \bullet_{33} \ar[d] \ar[ldd] & \circ_{34} \ar[d] \ar[ldd] & \circ_{35} \ar[d] \ar[ldd] \\ 
& \bullet_{20} \ar[d] \ar[ldd] & \circ_{21} \ar[d] \ar[ldd] & \circ_{22} \ar[d] \ar[ldd] & \bullet_{23} \ar[d] \ar[ldd] & & & \bullet_{27} \ar[d] \ar[ldd] & \circ_{28} \ar[d] \ar[ldd] & \circ_{29} \ar[d] \ar[ldd] \\
& \bullet_{14} \ar[d] \ar[ldd] & \circ_{15} \ar[d] \ar[ldd] & \circ_{16} \ar[d] \ar[ldd] & \bullet_{17} \ar[d] \ar[ldd] & & \bullet_{20} \ar[d] \ar[ldd] & \bullet_{21} \ar[d] \ar[ldd] & \circ_{22} \ar[d] \ar[ldd] & \circ_{23} \ar[d] \ar[ldd] \\
\bullet_7 \ar[d] & \bullet_8 \ar[d] & \circ_9 \ar[d] & \bullet_{10} \ar[d] & \bullet_{11} \ar[d] & & \bullet_{14} \ar[d] & \bullet_{15} \ar[d] & \circ_{16} \ar[d] & \bullet_{17} \ar[d] \\
\square_1 & \square_2 & \circ_3 & \bullet_4 & \bullet_5 & \square_7 & \square_8 & \square_9 & \circ_{10} & \bullet_{11}}
\caption{The Hasse diagram of $P_6$ is on the left. Any order ideal in $J_3 (P_6)$ must avoid the elements labelled by white circles, must contain elements labelled by squares, and might contain some of the elements labelled by black circles. A similar diagram for $J_4 (Q_5)$ is on the right.}
\end{center}
\end{figure} 

\begin{prop}
\label{tproposition}
For each integer $n \ge 0$,
\begin{flalign}
\label{tprecursion}
T_n^{(P)} = \displaystyle\sum_{i = 0}^{n - 1} \big( A_{n - i - 1} T_i^{(Q)} + i B_i A_{n - i - 1} + B_i T_{n - i - 1}^{(P)} \big)
\end{flalign}

\noindent and 
\begin{flalign}
\label{tqrecursion}
T_n^{(Q)} = \displaystyle\sum_{i = 0}^n \big( A_{n - i} T_i^{(P)} +  i A_i A_{n - i} + A_i T_{n - i}^{(P)} \big). 
\end{flalign}
\end{prop}

\begin{proof}
Let us first verify \hyperref[tprecursion]{(\ref*{tprecursion})}. Suppose that $i \in [1, n]$ is some integer; let $I \in J_i (P_n)$ be an order ideal. Then $I$ can be partitioned as the disjoint union $\{ 1, 2, \ldots , i - 1 \} \cup I_1 \cup I_2$, where $I_1 = I \cap P_n (i)$ and $I_2 = I \cap P_n (-i)$. Since $P_n (i) \simeq Q_{i - 1}$ and $P_n (-i) \simeq P_{n - i}$, we have that 
\begin{flalign*}
T_n^{(P)} &= \displaystyle\sum_{i = 1}^n \displaystyle\sum_{I \in J_i (P_n)} |I| \\
&= \displaystyle\sum_{i = 1}^n \displaystyle\sum_{I_1 \in J (Q_{i - 1})} \displaystyle\sum_{I_2 \in J (P_{n - i})} (i - 1 + |I_1| + |I_2|) \\
&= \displaystyle\sum_{i = 1}^n \big( T_{i - 1}^{(Q)} \big| J (P_{n - i}) \big| + (i - 1) \big| J(Q_{i - 1}) \big| \big| J (P_{n - i}) \big| + T_{n - i}^{(P)} \big| J(Q_{i - 1}) \big| \big),
\end{flalign*}

\noindent which implies \hyperref[tprecursion]{(\ref*{tprecursion})} since $\big| J (P_{n - i}) \big| = A_{n - i}$ and $\big| J(Q_{i - 1}) \big| = B_{i - 1}$. 

The proof of \hyperref[tqrecursion]{(\ref*{tqrecursion})} is analogous. Suppose that $i \in [1, n + 1]$ is an integer and let $I \in J_i (Q_n)$ be an order ideal. Then $I$ may be partitioned as the disjoint union $\{ n + 2, n + 3, \ldots , n + i \} \cup I_1 \cup I_2$, where $I_1 = I \cap Q_n (i)$ and $I_2 = I \cap Q_n (-i)$. Since $Q_n (i) \simeq P_{i - 1}$ and $Q_n (-i) \simeq P_{n - i + 1}$, we have that 
\begin{flalign*}
T_n^{(Q)} &= \displaystyle\sum_{i = 1}^{n + 1} \displaystyle\sum_{I \in J_i (Q_n)} |I| \\
&= \displaystyle\sum_{i = 1}^n \displaystyle\sum_{I_1 \in J (P_{i - 1})} \displaystyle\sum_{I_2 \in J (P_{n - i + 1})} (i - 1 + |I_1| + |I_2|) \\
&= \displaystyle\sum_{i = 1}^n \big( T_{i - 1}^{(P)} \big| J (P_{n - i + 1}) \big| + (i - 1) \big| J (P_{i - 1}) \big| \big| J(P_{n - i + 1}) \big| + T_{n - i + 1}^{(P)} \big| J (P_{i - 1}) \big| \big),
\end{flalign*}

\noindent which implies \hyperref[tqrecursion]{(\ref*{tqrecursion})}. 
\end{proof}

\noindent This yields a linear system of equations for the generating functions $T_P (x)$ and $T_Q (x)$ that can be solved explicitly. 

\begin{cor}
We have that 
\begin{flalign}
\label{tp}
T_P (x) = \displaystyle\frac{3x^2 A' (x)^2}{A(x)}
\end{flalign}

\noindent and 
\begin{flalign}
\label{tqrelation}
T_Q (x) = 2 A(x) T_P (x) + x A' (x) A(x). 
\end{flalign} 

\noindent Moreover, 
\begin{flalign}
\label{tqderivative}
T_Q' (x) = 2 A' (x) T_P (x) + 2 A(x) T_P' (x) + A' (x) A(x) + x A'' (x) A(x) + x A' (x)^2. 
\end{flalign}
\end{cor} 

\begin{proof}
The relation \hyperref[tqrelation]{(\ref*{tqrelation})} follows from \hyperref[tqrecursion]{(\ref*{tqrecursion})}. Differentiating \hyperref[tqrelation]{(\ref*{tqrelation})} yields \hyperref[tqderivative]{(\ref*{tqderivative})}. From \hyperref[tprecursion]{(\ref*{tprecursion})}, we deduce that
\begin{flalign}
\label{tprelation}
T_P (x) = x A(x) T_Q (x) + x^2 B' (x) A(x) + x B(x) T_P (x). 
\end{flalign}

\noindent By \hyperref[ab]{(\ref*{ab})}, $B' (x) = 2 A' (x) A(x)$; thus, inserting \hyperref[tqrelation]{(\ref*{tqrelation})} into \hyperref[tprelation]{(\ref*{tprelation})} yields 
\begin{flalign*}
T_P (x) = \displaystyle\frac{3x^2 A' (x) A(x)^2}{1 - 3 x A(x)^2}. 
\end{flalign*}

\noindent Applying \hyperref[derivative2]{(\ref*{derivative2})} to the above equality yields \hyperref[tp]{(\ref*{tp})}. 
\end{proof}

\noindent We may use a similar method to evaluate $R_P (x)$ and $R_Q (x)$. 

\begin{prop}
\label{rproposition}
For each integer $n \ge 0$, 

\begin{flalign}
\label{rprecursion}
R_n^{(P)} = \displaystyle\sum_{i = 0}^{n - 1} \big( A_{n - i - 1} R_i^{(Q)} + B_i R_{n - i - 1}^{(P)} \big)
\end{flalign}

\noindent and 

\begin{flalign}
\label{rqrecursion}
R_n^{(Q)} = \displaystyle\sum_{i = 0}^n \big( A_{n - i} (R_i^{(P)} + 2 T_i^{(P)}) + i A_i A_{n - i} + A_i (R_{n - i}^{(P)} + T_{n - i}^{(P)}) \big). 
\end{flalign}
\end{prop}

\begin{proof}
The proof is similar to that of \hyperref[tproposition]{Proposition \ref*{tproposition}}. Let us verify \hyperref[rqrecursion]{(\ref*{rqrecursion})} because the proof of \hyperref[rprecursion]{(\ref*{rprecursion})} is similar. Suppose that $i \in [1, n + 1]$ is an integer and let $I \in J_i (Q_n)$ be an order ideal. Then $I$ may be partitioned as the disjoint union $\{ n + 2, n + 3, \ldots , n + i \} \cup I_1 \cup I_2$, where $I_1 = I \cap Q_n (i)$ and $I_2 = I \cap Q_n (-i)$. Then, since $Q_n (i) \simeq P_{i - 1}$ and $Q_n (-i) \simeq P_{n - i + 1}$, we have that 
\begin{flalign*}
R_n^{(Q)} &= \displaystyle\sum_{i = 1}^{n + 1} \displaystyle\sum_{I \in J_i (Q_n)} \sum_{q \in I} \rho_{Q_n} (q) \\
&= \displaystyle\sum_{i = 1}^{n + 1} \displaystyle\sum_{I_1 \in J (P_{i - 1})} \displaystyle\sum_{I_2 \in J (P_{n - i + 1})} \left( i - 1 + \displaystyle\sum_{p_1 \in I_1} \big( \rho_{P_{i - 1}} (p_1) + 2 \big) + \displaystyle\sum_{p_2 \in I_2} \big( \rho_{P_{n - i + 1}} (p_2) + 1 \big) \right) \\
&= \displaystyle\sum_{i = 1}^{n + 1} \big( A_{n - i + 1} (R_{i - 1}^{(P)} + 2 T_{i - 1}^{(P)}) + (i - 1) A_{i - 1} A_{n - i + 1} + A_{i - 1} (R_{n - i + 1}^{(P)} + T_{n - i + 1}^{(P)}) \big), 
\end{flalign*}

\noindent which implies \hyperref[rqrecursion]{(\ref*{rqrecursion})}. 
\end{proof}

\begin{cor}
We have that 
\begin{flalign}
\label{rp}
R_P (x) = \displaystyle\frac{3 x A' (x) T_P (x) + x^2 A' (x)^2}{A(x)}
\end{flalign}

\noindent and 
\begin{flalign}
\label{rq}
R_Q (x) = 2 A(x) R_P (x) + 3 A(x) T_P (x) + x A' (x) A(x). 
\end{flalign}
\end{cor}

\begin{proof}
The relation \hyperref[rq]{(\ref*{rq})} follows from \hyperref[rqrecursion]{(\ref*{rqrecursion})}. From \hyperref[rprecursion]{(\ref*{rprecursion})}, we deduce that
\begin{flalign}
\label{rprelation}
R_P (x) = x A(x) R_Q (x) + x B(x) R_P (x). 
\end{flalign}

\noindent Inserting \hyperref[rq]{(\ref*{rq})} into \hyperref[rprelation]{(\ref*{rprelation})} and using \hyperref[ab]{(\ref*{ab})} yields that
\begin{flalign*}
R_P (x) = \displaystyle\frac{3x A(x)^2 T_P (x) + x^2 A' (x) A(x)^2}{1 - 3x A(x)^2}.
\end{flalign*}

\noindent Applying \hyperref[derivative2]{(\ref*{derivative2})} to the above gives \hyperref[rp]{(\ref*{rp})}. 
\end{proof}

\noindent We will now express $G_P (x)$ in terms of $A(x)$, $T_P (x)$, and $R_P (x)$. 

\begin{prop}
\label{gproposition}
For each integer $n \ge 0$,
\begin{flalign}
\label{sumidealsrecursion1} 
G_n^{(P)} &= \displaystyle\sum_{i = 0}^{n - 1} \Big( A_{n - i - 1} \big( G_i^{(Q)} + (n - i - 1) R_i^{(Q)} - i T_i^{(Q)} \big) \nonumber \\
& \qquad + B_i \big( G_{n - i - 1}^{(P)} + (i + 1) R_{n - i - 1}^{(P)} + T_{n - i - 1}^{(P)} \big) \nonumber \\
& \qquad + i B_i A_{n - i - 1} - T_i^{(Q)} T_{n - i - 1}^{(P)} \Big)
\end{flalign}

\noindent and 
\begin{flalign}
\label{sumidealsrecursion2}
G_n^{(Q)} &= \displaystyle\sum_{i = 0}^n \Big( A_{n - i} \big(G_i^{(P)} + (n - i + 1) R_i^{(P)} + (2n + 3 - i) T_i^{(P)} \big) \nonumber \\
& \qquad + A_i \big( G_{n - i}^{(P)} + (i + 1) R_{n - i}^{(P)} + (n + 2) T_{n - i}^{(P)} \big) \nonumber \\
& \qquad + (n + 2) i A_i A_{n - i} - T_i^{(P)} T_{n - i}^{(P)} \Big). 
\end{flalign}
\end{prop}

\begin{proof}
Again the proof is similar to the proofs of \hyperref[tproposition]{Proposition \ref*{tproposition}} and \hyperref[rproposition]{Proposition \ref*{rproposition}}. We will verify \hyperref[sumidealsrecursion2]{(\ref*{sumidealsrecursion2})} since the proof of \hyperref[sumidealsrecursion1]{(\ref*{sumidealsrecursion1})} is similar. Let $i \in [1, n + 1]$ be an integer and $I \in J_i (Q_n)$ be an order ideal. Using the decomposition $I = \{ n + 2, n + 3, \ldots , n + i \} \cup I_1 \cup I_2$ as before (and the isomorphisms $Q_n (i) \simeq P_{i - 1}$ and $Q_n (-i) \simeq P_{n - i + 1}$), one may check that 
\begin{flalign}
\label{sigmasum2}
\displaystyle\sum_{I \in J (Q_n)} \sigma(I) &= \displaystyle\sum_{i = 1}^{n + 1} \displaystyle\sum_{I_1 \in J (P_{i - 1})} \displaystyle\sum_{I_2 \in J (P_{n - i + 1})} \Bigg( \displaystyle\sum_{j = n + 2}^{n + i} j + \displaystyle\sum_{p_1 \in I_1} \big( p_1 + (n + 2 - i) \rho_{P_{i - 1}} (p_1) + 2n + 3 \big) \nonumber \\
& \qquad \qquad \qquad \qquad \qquad \qquad \quad + \displaystyle\sum_{p_2 \in I_2} \big( p_2 + i \rho_{P_{n - i + 1}} (p_2) + n + i + 1 \big) \Bigg). 
\end{flalign}

\noindent Furthermore, 
\begin{flalign}
\label{lengthsum2}
\displaystyle\sum_{I \in J(Q_n)} \binom{|I|}{2} &= \displaystyle\sum_{i = 1}^{n + 1} \displaystyle\sum_{I_1 \in J (P_{i - 1})} \displaystyle\sum_{I_2 \in J (P_{n - i + 1})} \binom{|I_1| + |I_2| + i - 1}{2} \nonumber \\
&= \displaystyle\sum_{i = 1}^{n + 1} \displaystyle\sum_{I_1 \in J (P_{i - 1})} \displaystyle\sum_{I_2 \in J (P_{n - i + 1})} \Bigg( \binom{|I_1|}{2} + \binom{|I_2|}{2} + |I_1| |I_2| + (i - 1) |I_1| \nonumber \\
& \qquad \qquad \qquad \qquad \qquad \qquad \quad + (i - 1) |I_2| + \binom{i - 1}{2} \bigg). 
\end{flalign}

\noindent Subtracting \hyperref[lengthsum2]{(\ref*{lengthsum2})} from \hyperref[sigmasum2]{(\ref*{sigmasum2})} yields 
\begin{flalign*}
G_n^{(Q)} &= \displaystyle\sum_{i = 1}^{n + 1} \displaystyle\sum_{I_1 \in J (P_{i - 1})} \displaystyle\sum_{I_2 \in J (P_{n - i + 1})} \Bigg( (i - 1)(n + 2) + \displaystyle\sum_{p_1 \in I_1} \big( (n + 2 - i) \rho_{P_{i - 1}} (p_1) + 2n + 4 - i \big) + \sigma(I_1) \\
& \qquad \qquad \qquad \qquad \qquad - \binom{|I_1|}{2} + \displaystyle\sum_{p_2 \in I_2} \big( i \rho_{P_{n - i + 1}} (p_2) + n + 2 \big) + \sigma(I_2) - \binom{|I_2|}{2} - |I_1| |I_2| \Bigg), 
\end{flalign*} 

\noindent which implies \hyperref[sumidealsrecursion2]{(\ref*{sumidealsrecursion2})}. 
\end{proof}

\begin{cor}
We have that 
\begin{flalign}
\label{g2}
G_P (x) &= \big( 3x A(x)^2 R_P (x) + 6 x^2 A' (x) A(x) R_P (x) + 6 x A(x)^2 T_P (x) \nonumber \\
& \quad + 3 x^2 A' (x) A(x) T_P (x) + 4 x^2 A' (x) A(x)^2 + x^3 A' (x)^2 A(x) \nonumber \\
& \quad - 3 x A(x) T_P (x)^2 \big) (1 - 3x A(x)^2)^{-1}. 
\end{flalign}
\end{cor}
\begin{proof}
From \hyperref[sumidealsrecursion2]{(\ref*{sumidealsrecursion2})} and \hyperref[sumidealsrecursion1]{(\ref*{sumidealsrecursion1})}, we deduce that 
\begin{flalign*}
G_Q (x) &= 2 A(x) G_P (x) + 2 A(x) R_P (x) + 2 x A' (x) R_P (x) + 5 A(x) T_P (x) + 3 x A' (x) T_P (x) \\
& \quad + 2 x A(x) T_P' (x) + 3x A' (x) A(x) + x^2 A' (x)^2 + x^2 A'' (x) A(x) - T_P (x)^2 
\end{flalign*}

\noindent and 
\begin{flalign*}
G_P (x) &= x A(x) G_Q (x) + x^2 A' (x) R_Q (x) - x^2 A(x) T_Q' (x) + x B(x) G_P (x) + x B(x) R_P (x) \\
& \quad + x^2 B' (x) R_P (x) + x B(x) T_P (x) + x^2 B' (x) A(x) - x T_P (x) T_Q (x), 
\end{flalign*}

\noindent respectively. Inserting the first equality above into the second and using \hyperref[ab]{(\ref*{ab})} gives 
\begin{flalign*}
G_P (x) &= 3 x A(x)^2 G_P (x) + 3 x A(x)^2 R_P (x) + 4 x^2 A' (x) A(x) R_P (x) + x^2 A' (x) R_Q (x) \\
& \quad + 6 x A(x)^2 T_P (x) + 3 x^2 A' (x) A(x) T_P (x) + 2 x^2 A(x)^2 T_P' (x) - x^2 A(x) T_Q' (x) \\
& \quad + 5 x^2 A' (x) A(x)^2  + x^3 A' (x)^2 A(x) + x^3 A'' (x) A(x)^2 - x A(x) T_P (x)^2 - x T_P (x) T_Q (x). 
\end{flalign*}

\noindent Applying \hyperref[tqrelation]{(\ref*{tqrelation})}, \hyperref[tqderivative]{(\ref*{tqderivative})}, and \hyperref[rq]{(\ref*{rq})} to the above yields \hyperref[g2]{(\ref*{g2})}. 
\end{proof}

\noindent We may now prove \hyperref[idealtheorem2]{Theorem \ref*{idealtheorem2}}. 

\begin{proof}[Proof of Theorem 2.1] 
As noted previously, it suffices to establish \hyperref[function2]{(\ref*{function2})}. The right side of this equality involves derivatives of $A(x)$. We can express these in terms of $A(x)$ using \hyperref[relation2]{(\ref*{relation2})}. Specifically, differentiating \hyperref[relation2]{(\ref*{relation2})} yields 
\begin{flalign}
\label{derivative2} 
A' (x) = \displaystyle\frac{A(x)^3}{1 - 3x A(x)^2}. 
\end{flalign} 

\noindent Differentiating again gives 
\begin{flalign}
\label{secondderivative2}
A'' (x) = \displaystyle\frac{3 A(x)^2 \big( A' (x) + A(x)^3 -  x A(x)^2 A' (x) \big)}{(1 - 3x A(x)^2)^2}
\end{flalign}

\noindent and repeating yields  
\begin{flalign} 
\label{thirdderivative2}
A''' (x) &= 3 A(x) \big(3x^2 A(x)^5 A'' (x) - 4x A(x)^3 A'' (x) + A(x) A'' (x) - 6x A(x)^5 A' (x) \nonumber \\
& \quad + 10 A(x)^3 A' (x) + 2x A(x)^2 A' (x)^2 + 2 A' (x)^2 + 6 A(x)^6 \big) (1 - 3x A(x)^2)^{-3}. 
\end{flalign}

\noindent Now, in order to establish \hyperref[function2]{(\ref*{function2})}, apply \hyperref[g2]{(\ref*{g2})}, \hyperref[derivative2]{(\ref*{derivative2})}, \hyperref[tp]{(\ref*{tp})}, and \hyperref[rp]{(\ref*{rp})} to express the left side as a rational function in $x$ and $A(x)$. Applying \hyperref[secondderivative2]{(\ref*{secondderivative2})}, \hyperref[thirdderivative2]{(\ref*{thirdderivative2})}, and \hyperref[derivative2]{(\ref*{derivative2})}, we also express the right side as a rational function in $x$ and $A(x)$. Simplifying, we obtain that that the two sides are equal; we omit this computation here (but the proof of a more general identity may be found at the end of Section 3). 
\end{proof}

\section{Proof of Theorem 1.3} 
In this section, we will prove \hyperref[idealtheoremgeneral]{Theorem \ref*{idealtheoremgeneral}} through a method similar to the one used when $m = 2$. We will suppose that $m > 1$, since the case $m = 1$ has been established by Stanley and Zanello \cite{14}. Let us begin by defining several posets. For each nonnegative integer $n$, let $P_n = P_n^{(0)} = P_{n, mn + 1}$. For each integer $j \in [1, m - 1]$, let $P_n^{(j)}$ be the poset obtained from removing the elements $p \in P_{n + 1}^{(0)}$ with $\lfloor p / (n + 1) \rfloor < j$; equivalently, $P_n^{(j)} = P_{n + 1} \backslash \bigcup_{h = 0}^{j - 1} \{ h(n + 1) + 1, h(n + 1) + 2, \ldots , h(n + 1) + n \}$.  If $m = 2$, then observe that $P_n^{(1)} = Q_n$ from the previous section. For each nonnegative integer $n$, let $A_n$ denote the number of order ideals in $P_n$; for each $j \in [0, m - 1]$, let $A_n^{(j)}$ be the number of order ideals in $P_n^{(j)}$. Applying the theorem of Bizley (see \cite{6}), we see that
\begin{flalign}
\label{expressionageneral}
A_n^{(0)} = \binom{mn + n + 1}{n} / (mn + n + 1).
\end{flalign} 

Define the generating function $A^{(j)} (x) = \sum_{k = 0}^{\infty} A_k^{(j)} x^k$, where $x$ is a formal variable; let $A(x) = A^{(0)} (x)$. In order to obtain analogues of \hyperref[ab]{(\ref*{ab})} and \hyperref[relation2]{(\ref*{relation2})}, we will apply a recursive method similar to the one used in the previous section. 

For each integer $i \in [1, n]$, let $J_i (P_n) \subset J (P_n)$ be the set of order ideals of $P_n$ that contain $\{ 1, 2, \ldots , i - 1 \}$ but not $i$. For each integer $i \in [1, n + 1]$ and $h \in [1, m - 1]$, let $J_i \big( P_n^{(h)} \big) \subset J \big( P_n^{(h)} \big)$ denote the set of order ideals of $P_n^{(h)}$ that contain $\{ h(n + 1) + 1, h(n + 1) + 2, \ldots , h(n + 1) + i - 1 \}$ but not $h(n + 1) + i$. When $m = 2$ and $h = 1$, we recover $J_i (Q_n)$ from the previous section. As in Section 2, we may partition $J (P_n) = \bigcup_{i = 1}^n J_i (P_n)$ and $J \big( P_n^{(h)} \big) = \bigcup_{i = 1}^{n + 1} J_i \big( P_n^{(h)} \big)$. We will use these decompositions to obtain the following result. 

\begin{prop}
\label{relationgeneral}
We have that $A^{(j)} (x) = A(x)^{m - j + 1}$ for each integer $j \in [1, m - 1]$. Moreover, $x A(x)^{m + 1} - A(x) + 1 = 0$. 
\end{prop}

\begin{proof}
To verify the first equality, it suffices to check that $A^{(h)} (x) = A(x) A^{(h + 1)} (x)$ for each integer $h \in [1, m - 1]$, where the index $h$ is taken modulo $m$. Let $i \in [1, n + 1]$ and $h \in [1, m - 1]$ be integers and let $I \in J_i \big( P_n^{(h)} \big)$ be an order ideal. As in the previous section, $I$ can be partitioned as the disjoint union $\{ h(n + 1) + 1, h(n + 1) + 2, \ldots , h(n + 1) + i - 1 \} \cup I_1 \cup I_2$, where $I_1$ consists of the elements of $I$ greater than some $j \in \{ h(n + 1) + 1, h(n + 1) + 2, \ldots , h(n + 1) + i - 1 \}$ and incomparable to each $j \in \{ h(n + 1) + i, h(n + 1) + i + 2, \ldots , h(n + 1) + n \}$ (where the ordering is with respect to the poset $P_n^{(h)}$) and $I_2$ consists of the elements of $I$ that are incomparable to each $j \in \{ h(n + 1) + 1, h(n + 1) + 2, \ldots , h(n + 1) + i \}$. Observe that $I_1$ is an order ideal in a poset isomorphic to $P_{i - 1}^{(h + 1)}$ and that $I_2$ is an order ideal in a poset isomorphic to $P_{n - i + 1}$. Hence,
\begin{flalign*}
A_n^{(h)} &= \displaystyle\sum_{i = 1}^{n + 1} \displaystyle\sum_{I \in J_i (P_n^{(h)})} 1 = \displaystyle\sum_{i = 1}^{n + 1} \displaystyle\sum_{I_1 \in J (P_{i - 1}^{(h + 1)})} \displaystyle\sum_{I_2 \in J(P_{n - i + 1})} 1 = \displaystyle\sum_{i = 0}^n A_i^{(h + 1)} A_{n - i}. 
\end{flalign*}

\noindent This recursion yields the relation $A^{(h)} (x) = A(x) A^{(h + 1)} (x)$ for all integers $h \in [1, m - 1]$, thereby establishing the first statement of the proposition. The second statement of the proposition follows from the equality $A(x) = x A^{(1)} (x) A(x) + 1$, which can be verified through a similar recursive method. 
\end{proof}

\noindent For each integer $n \ge 0$ and each $p \in P_n$, let $\rho_{n, 0} (p) = \lfloor p / n \rfloor$. For each integer $j \in [1, m - 1]$ and element $q \in P_n^{(j)}$, let $\rho_{n, j} (q) = \rho_{n, 0} (q)$. For each integer $j \in [0, m - 1]$, define the sums  
\begin{flalign*} 
T_n^{(j)} = \displaystyle\sum_{I \in J (P_n^{(j)})} |I|; \quad R_n^{(j)} = \displaystyle\sum_{I \in J(P_n^{(j)})} \displaystyle\sum_{i \in I} \rho_{n, j} (i); \quad G_n^{(j)} = \displaystyle\sum_{I \in J(P_n^{(j)})} \left( \sigma(I) - \binom{|I|}{2} \right). 
\end{flalign*}

\noindent Also define the generating functions 
\begin{flalign*}
T_j (x) = \displaystyle\sum_{k = 0}^{\infty} T_k^{(j)} x^k; \quad R_j (x) = \displaystyle\sum_{k = 0}^{\infty} R_k^{(j)} x^k; \quad G_j (x) = \displaystyle\sum_{k = 0}^{\infty} G_k^{(j)} x^k. 
\end{flalign*}

\noindent Analogous to \hyperref[idealtheorem2]{Theorem \ref*{idealtheorem2}} (which is equivalent to \hyperref[function2]{(\ref*{function2})}), \hyperref[idealtheoremgeneral]{Theorem \ref*{idealtheoremgeneral}} is equivalent to an algebraic identity involving $G_0 (x)$ and derivatives of $A(x)$. Specifically, due to \hyperref[expressionageneral]{(\ref*{expressionageneral})}, it suffices to establish the equality 
\begin{flalign}
\label{functiongeneral}
m (m + 1) x^3 A''' (x) + m (2m + 4) x^2 A'' (x) - 24 G_0 (x) = 0 
\end{flalign}

\noindent in order to prove \hyperref[idealtheoremgeneral]{Theorem \ref*{idealtheoremgeneral}}. As in Section 2, we will deduce \hyperref[functiongeneral]{(\ref*{functiongeneral})} by expressing $T_j (x)$, $R_j (x)$, and $G_0 (x)$ as rational functions in $x$ and $A(x)$. Let us begin with $T_j (x)$. 

\begin{prop}
For each integer $j \in [1, m - 2]$, 
\begin{flalign}
\label{tjrelation}
T_j (x) &= A(x) T_{j + 1} (x) + (m - j) x A' (x) A(x)^{m - j} + A(x)^{m - j} T_0 (x). 
\end{flalign}

\noindent Moreover, 
\begin{flalign}
\label{t0relation}
T_0 (x) &= x A(x) T_1 (x) + m x^2 A' (x) A(x)^m + x A(x)^m T_0 (x)
\end{flalign}

\noindent and 
\begin{flalign}
\label{trelation}
T_{m - 1} (x) &= 2 A(x) T_0 (x) + x A' (x) A(x). 
\end{flalign}

\end{prop}

\begin{proof}
Following the proof of \hyperref[tproposition]{Proposition \ref*{tproposition}}, one obtains that
\begin{flalign*}T_n^{(j)} = \displaystyle\sum_{i = 0}^n \big( T_i^{(j + 1)} A_{n - i} + i A_i^{(j + 1)} A_{n - i} + A_i^{(j + 1)} T_{n - i}^{(0)} \big). 
\end{flalign*}

\noindent for each integer $j \in [1, m - 2]$;  
\begin{flalign*}
T_n^{(0)} = \displaystyle\sum_{i = 0}^{n - 1} \big( T_i^{(1)} A_{n - i - 1} + i A_i A_{n - i - 1} + A_i^{(1)} T_{n - i - 1}^{(0)} \big); 
\end{flalign*}

\noindent and 
\begin{flalign*}
T_n^{(m - 1)} = \displaystyle\sum_{i = 0}^n \big( T_i^{(0)} A_{n - i} + i A_i A_{n - i} + A_i T_{n - i}^{(0)} \big). 
\end{flalign*}

\noindent These recursive relations imply the proposition. 
\end{proof} 

\begin{cor}
For each integer $j \in [1, m - 1]$, 
\begin{flalign}
\label{tj}
T_j (x) = (m + 1 - j) A(x)^{m - j} T_0 (x) + \binom{m + 1 - j}{2} x A' (x) A(x)^{m - j}
\end{flalign}

\noindent and 
\begin{flalign}
\label{t0}
T_0 (x) = \displaystyle\frac{\binom{m + 1}{2} x^2 A' (x)^2}{A(x)}. 
\end{flalign}

\noindent Moreover, 
\begin{flalign}
\label{tsum}
\displaystyle\sum_{j = 1}^{m - 1} A(x)^{j - 1} T_j (x) &= A(x)^{m - 1} \left( \displaystyle\frac{(m^2 + m - 2) T_0 (x)}{2} + \binom{m + 1}{3} x A' (x) \right).  
\end{flalign} 
\end{cor} 

\begin{proof}
Using \hyperref[tjrelation]{(\ref*{tjrelation})} and induction on $m - j$ (the base case $m - j = 1$ is given by \hyperref[trelation]{(\ref*{trelation})}), we obtain \hyperref[tj]{(\ref*{tj})}. Multiplying \hyperref[tj]{(\ref*{tj})} by $A(x)^{j - 1}$ and summing over $j$ yields \hyperref[tsum]{(\ref*{tsum})}. Inserting \hyperref[tj]{(\ref*{tj})}, with $j = 1$, into \hyperref[t0relation]{(\ref*{t0relation})} gives \hyperref[t0]{(\ref*{t0})}. 

\end{proof}

\begin{cor}
We have that 
\begin{flalign}
\label{tderivativesum} 
\displaystyle\sum_{j = 1}^{m - 1} A(x)^j T_j' (x) = A(x)^{m - 1} &\Bigg( \displaystyle\frac{(m^2 + m - 2) A(x) T_0' (x)}{2} + \displaystyle\frac{(m - 1) m (m + 1) A' (x) T_0 (x)}{3} \nonumber \\
& \quad  + \binom{m + 1}{3} A' (x) A(x) + \binom{m + 1}{3} x A'' (x) A(x) \nonumber \\
& \quad + \displaystyle\frac{(m - 1) m (m + 1) (3m - 2) x A' (x)^2}{24} \Bigg). 
\end{flalign}
\end{cor}

\begin{proof}
Differentiating \hyperref[tj]{(\ref*{tj})} gives 
\begin{flalign*}
T_j' (x) &= (m + 1 - j) A(x)^{m - j} T_0' (x) + (m + 1 - j)(m - j) A' (x) A(x)^{m - j - 1} T_0 (x) \nonumber \\
& \quad + \binom{m + 1 - j}{2} \big( A' (x) A(x)^{m - j} + x A'' (x) A(x)^{m - j} + (m - j) x A' (x)^2 A(x)^{m - j - 1} \big) 
\end{flalign*}

\noindent for each integer $j \in [1, m - 1]$. Multiplying this equality by $A(x)^j$ and summing over $j$ yields \hyperref[tderivativesum]{(\ref*{tderivativesum})}. 
\end{proof} 

\noindent Next, we will find $R_j (x)$. 
\begin{prop}
For each integer $j \in [1, m - 2]$, 
\begin{flalign}
\label{rjrelation}
R_j (x) = A(x) R_{j + 1} (x) + j A(x)^{m - j} T_0 (x) + j (m - j) x A' (x) A(x)^{m - j} + A(x)^{m - j} R_0 (x).
\end{flalign}

\noindent Moreover, 
\begin{flalign}
\label{r0relation}
R_0 (x) = x A(x) R_1 (x) + x A(x)^m R_0 (x)
\end{flalign}

\noindent and
\begin{flalign}
\label{rrelation}
R_{m - 1} (x) = 2 A(x) R_0 (x) + (2m - 1) A(x) T_0 (x) + (m - 1) x A' (x) A(x). 
\end{flalign}
\end{prop}

\begin{proof}
Following the proof of \hyperref[rproposition]{Proposition \ref*{rproposition}}, one obtains that 
\begin{flalign*}
R_n^{(j)} = \displaystyle\sum_{i = 0}^n \big( A_{n - i} R_i^{(j + 1)} + ij A_i^{(j + 1)} A_{n - i} + A_i^{(j + 1)} (R_{n - i}^{(0)} + j T_{n - i}^{(0)}) \big) 
\end{flalign*}

\noindent for each integer $j \in [1, m - 2]$; 
\begin{flalign*}
R_n^{(0)} = \displaystyle\sum_{i = 0}^{n - 1} \big( A_{n - i - 1} R_i^{(1)} + A_i^{(1)} R_{n - i - 1}^{(0)} \big); 
\end{flalign*}

\noindent and 

\begin{flalign*}
R_n^{(m - 1)} = \displaystyle\sum_{i = 0}^{n - 1} \big( A_{n - i} (R_i^{(0)} + m T_i^{(0)}) + i (m - 1) A_i A_{n - i} + A_i (R_{n - i}^{(0)} + (m - 1) T_{n - i}^{(0)}) \big). 
\end{flalign*}

\noindent These recursive relations imply the proposition. 
\end{proof}

\begin{cor}
We have that
\begin{flalign}
\label{r0}
R_0 (x) = \displaystyle\frac{\binom{m + 1}{2} x A' (x) T_0 (x) + \binom{m + 1}{3} x^2 A' (x)^2}{A(x)}
\end{flalign}

\noindent and 
\begin{flalign}
\label{rsum}
\displaystyle\sum_{j = 1}^{m - 1} A(x)^{j - 1} R_j (x) = A(x)^{m - 1} &\Bigg( \displaystyle\frac{(m^2 + m - 2) R_0 (x)}{2} + \displaystyle\frac{m(2m^2 + 3m - 5) T_0 (x)}{6} \nonumber \\
& \quad  + \displaystyle\frac{(m - 1) m^2 (m + 1) x A' (x)}{12} \Bigg). 
\end{flalign}
\end{cor}

\begin{proof}
Using \hyperref[rjrelation]{(\ref*{rjrelation})} and induction on $m - j$ (the base case $m - j = 1$ is given by \hyperref[rrelation]{(\ref*{rrelation})}), we obtain that 
\begin{flalign}
\label{rj}
R_j (x) &= (m - j + 1) A(x)^{m - j} R_0 (x) + \displaystyle\frac{(m - j + 1)(m + j) A(x)^{m - j} T_0 (x)}{2} \nonumber \\
& \quad + \displaystyle\frac{m + 2j - 1}{3} \binom{m - j + 1}{2} x A' (x) A(x)^{m - j}
\end{flalign}

\noindent for each integer $j \in [1, m - 1]$. Multiplying \hyperref[rj]{(\ref*{rj})} by $A(x)^{j - 1}$ and summing over $j$ yields \hyperref[rsum]{(\ref*{rsum})}. Inserting \hyperref[rj]{(\ref*{rj})}, with $j = 1$, into \hyperref[r0relation]{(\ref*{r0relation})} gives \hyperref[r0]{(\ref*{r0})}. 
\end{proof}

\noindent We may now evaluate $G_0 (x)$. 
\begin{prop}
\label{grelation}
For each integer $j \in [1, m - 2]$, 
\begin{flalign*}
G_i (x) &= A(x) G_{i + 1} (x) + x A' (x) R_{i + 1} (x) - x A(x) T_{i + 1}' (x) \\
& \quad + A(x)^{m - i} G_0 (x) + A(x)^{m - i} R_0 (x) + (m - i) x A' (x) A(x)^{m - i - 1} R_0 (x) \\
& \quad + (i + 1) A(x)^{m - i} T_0 (x) + i (m - i) x A' (x) A(x)^{m - i - 1} T_0 (x) + i x A(x)^{m - i} T_0' (x) \\
& \quad + (2i + 1)(m - i) x A' (x) A(x)^{m - i} + i (m - i) x^2 A'' (x) A(x)^{m - i} \\
& \quad + i (m - i)^2 x^2 A' (x)^2 A(x)^{m - i - 1} - T_0 (x) T_{i + 1} (x). 
\end{flalign*}

\noindent Moreover, 
\begin{flalign*}
G_0 (x) &= x A(x) G_1 (x) + x^2 A' (x) R_1 (x) - x^2 A(x) T_1' (x) \\
& \quad + x A(x)^m G_0 (x) + x A(x)^m R_0 (x) + m x^2 A' (x) A(x)^{m - 1} R_0 (x) \\
& \quad + x A(x)^m T_0 (x) + m x^2 A' (x) A(x)^m - x T_0 (x) T_1 (x),
\end{flalign*}

\noindent and 
\begin{flalign*}
G_{m - 1} (x) &= 2 A(x) G_0 (x) + 2 x A' (x) R_0 (x) + 2 A(x) R_0 (x) \\
& \quad + (2m + 1) A(x) T_0 (x) + (2m - 1) x A' (x) T_0 (x) + (2m - 2) x A(x) T_0' (x) \\
& \quad + (2m - 1) x A' (x) A(x) + (m - 1) x^2 A'' (x) A(x) \\
& \quad + (m - 1) x^2 A' (x)^2 - T_0 (x)^2.  
\end{flalign*}
\end{prop}

\begin{proof}
Following the proof of \hyperref[gproposition]{Proposition \ref*{gproposition}}, one obtains that 
\begin{flalign*}
G_n^{(j)} &= \displaystyle\sum_{i = 0}^n \Big( A_{n - i} \big( G_i^{(j + 1)} + (n - i) R_i^{(j + 1)} - i T_i^{(j + 1)} \big) \\
& \qquad + A_i^{(j + 1)} \big( G_{n - i}^{(0)} + (i + 1) R_{n - i}^{(0)} + (j (n + 1) + 1) T_{n - i}^{(0)} \big) \\
& \qquad + (j (n + 1) + 1) i A_i^{(j + 1)} A_{n - i} - T_i^{(j + 1)} T_{n - i}^{(0)} \Big)
\end{flalign*}

\noindent for each integer $j \in [1, m - 2]$;
\begin{flalign*}
G_n^{(0)} &= \displaystyle\sum_{i = 0}^n \Big( A_{n - i - 1} \big( G_i^{(1)} + (n - i - 1) R_i^{(1)} - i T_i^{(1)} \big) \\
& \qquad + A_i^{(1)} \big( G_{n - i - 1}^{(0)} + (i + 1) R_{n - i - 1}^{(0)} + T_{n - i - 1}^{(0)} \big) \\
& \qquad + i A_i^{(1)} A_{n - i - 1} - T_i^{(1)} T_{n - i - 1}^{(0)} \Big); 
\end{flalign*}

\noindent and
\begin{flalign*}
G_n^{(m - 1)} &= \displaystyle\sum_{i = 0}^n \Big( A_{n - i} \big( G_i^{(0)} + (n - i + 1) R_i^{(0)} + (m(n + 1) + 1 - i) T_i^{(0)} \big) \\
& \qquad + A_i \big( G_{n - i}^{(0)} + (i + 1) R_{n - i}^{(0)} + ((m - 1)(n + 1) + 1) T_{n - i}^{(0)} \big) \\
& \qquad + ((m - 1)(n + 1) + 1) i A_i A_{n - i} - T_i^{(0)} T_{n - i}^{(0)} \Big).
\end{flalign*}

\noindent These recursive relations imply the proposition. 
\end{proof}

\begin{cor}
We have that 
\begin{flalign}
\label{g0}
G_0 (x) &= \bigg( (m + 1) x A(x)^m R_0 (x) + (m^2 + m) x^2 A' (x) A(x)^{m - 1} R_0 (x) \nonumber \\
& \quad + \binom{m + 2}{2} x A(x)^m T_0 (x) + \binom{m + 1}{2} x^2 A' (x) A(x)^{m - 1} T_0 (x) \nonumber \\
& \quad + \binom{m + 2}{3} x^2 A' (x) A(x)^m + \binom{m + 2}{4} x^3 A' (x)^2 A(x)^{m - 1} \nonumber \\
& \quad - \binom{m + 1}{2} x A(x)^{m - 1} T_0 (x)^2 \bigg) (1 - (m + 1) x A(x)^m)^{-1}. 
\end{flalign}
\end{cor}

\begin{proof}
Using \hyperref[grelation]{Proposition \ref*{grelation}} and the equality 
\begin{flalign*}
G_0 (x) = \big( G_0 (x) - x A(x) G_1 (x) \big) + x \displaystyle\sum_{j = 1}^{m - 2} \big( A(x)^j G_j (x) - A(x)^{j + 1} G_{j + 1} (x) \big) + x A(x)^{m - 1} G_{m - 1} (x)
\end{flalign*}

\noindent gives 
\begin{flalign*}
G_0 (x) &= \Bigg( (m + 1) x A(x)^m R_0 (x) + \displaystyle\frac{(m^2 + m + 2) x^2 A' (x) A(x)^{m - 1} R_0 (x)}{2} \\
& \quad + x^2 A' (x) \displaystyle\sum_{j = 1}^{m - 1} A(x)^{j - 1} R_j (x) - x^2 \displaystyle\sum_{j = 1}^{m - 1} A(x)^j T_j' (x) \\
& \quad + \binom{m + 2}{2} x A(x)^m T_0 (x) + \displaystyle\frac{m(m^2 + 5) x^2 A' (x) A(x)^{m - 1} T_0 (x)}{6} \\
& \quad + \displaystyle\frac{(m + 2)(m - 1) x^2 A(x)^m T_0' (x)}{2}  + \displaystyle\frac{m(2m + 1)(m + 1) x^2 A' (x) A(x)^m}{6} \\
& \quad + \binom{m + 1}{3} x^3 A'' (x) A(x)^m + \displaystyle\frac{(m - 1) m^2 (m + 1) x^3 A' (x)^2 A(x)^{m - 1}}{12} \\
& \quad - x A(x)^{m - 1} T_0 (x)^2 - x T_0 (x) \displaystyle\sum_{j = 1}^{m - 1} A(x)^{j - 1} T_j (x) \Bigg) (1 - (m + 1) x A(x)^m)^{-1}.  
\end{flalign*}

\noindent Applying \hyperref[tsum]{(\ref*{tsum})}, \hyperref[tderivativesum]{(\ref*{tderivativesum})}, and \hyperref[rsum]{(\ref*{rsum})} to the above yields  \hyperref[g0]{(\ref*{g0})}. 
\end{proof}

\noindent We may now prove \hyperref[idealtheoremgeneral]{Theorem \ref*{idealtheoremgeneral}}. 

\begin{proof}[Proof of Theorem 1.3]
As stated previously, it suffices to establish \hyperref[functiongeneral]{(\ref*{functiongeneral})}. The left side of this equality involves derivatives of $A(x)$, which can be expressed in terms of $A(x)$ using \hyperref[relationgeneral]{Proposition \ref*{relationgeneral}}. Specifically, differentiating the second equality stated in \hyperref[relationgeneral]{Proposition \ref*{relationgeneral}}, we obtain that
\begin{flalign} 
\label{derivativegeneral}
A' (x) = \displaystyle\frac{A(x)^{m + 1}}{1 - (m + 1) x A(x)^m}. 
\end{flalign}

\noindent Differentiating again yields 
\begin{flalign}
\label{secondderivativegeneral}
A'' (x) = \displaystyle\frac{(m + 1) A(x)^m \big( A' (x) + A(x)^{m + 1} - x A(x)^m A'(x) \big)}{(1 - (m + 1) x A(x)^m)^2} 
\end{flalign}

\noindent and repeating gives 
\begin{flalign}
\label{thirdderivativegeneral}
A''' (x) &= (m + 1) A(x)^{m - 1} \big( A(x) A'' (x) + (m - 1) m x A(x)^m A' (x)^2 + (4m + 2) A(x)^{m + 1} A' (x) \nonumber \\
& \quad + m A' (x)^2 - (m + 2) x A(x)^{m + 1} A'' (x) + (m + 1) x^2 A(x)^{2m + 1} A'' (x) \nonumber \\
& \quad - 2(m + 1) x A' (x) A(x)^{2m + 1} + 2 (m + 1) A(x)^{2m + 2} \big) \big( 1 - (m + 1) x A(x)^m)^{-3}. 
\end{flalign}

\noindent Now, we may express the left side of \hyperref[functiongeneral]{(\ref*{functiongeneral})} as a rational function in $x$ and $A(x)$ using \hyperref[g0]{(\ref*{g0})}, \hyperref[derivativegeneral]{(\ref*{derivativegeneral})}, \hyperref[secondderivativegeneral]{(\ref*{secondderivativegeneral})}, \hyperref[thirdderivativegeneral]{(\ref*{thirdderivativegeneral})}, \hyperref[t0]{(\ref*{t0})}, and \hyperref[r0]{(\ref*{r0})}. After inserting these identities into the left side and simplifying, one obtains $0$. The below Sage code verifies this claim since its output is $0$. 

\noindent A,m=var(`A',`m')

\noindent A1=A\textasciicircum (m+1)/(1-(m+1)*x*A\textasciicircum m)

\noindent A2=(m+1)*A\textasciicircum m*(A1+A\textasciicircum (m+1)-x*A\textasciicircum m*A1)/(1-(m+1)*x*A\textasciicircum m)\textasciicircum 2

\noindent A3=(m+1)*A\textasciicircum (m-1)*(A*A2+(m-1)*m*x*A\textasciicircum m*A1\textasciicircum 2+(4*m+2)*A\textasciicircum (m+1)*A1+m*A1\textasciicircum 2

\noindent -(m+2)*x*A\textasciicircum (m+1)*A2+(m+1)*x\textasciicircum 2*A\textasciicircum (2*m+1)*A2-2*(m+1)*x*A\textasciicircum (2*m+1)*A1

\noindent +2*(m+1)*A\textasciicircum (2*m+2))/(1-(m+1)*x*A\textasciicircum m)\textasciicircum 3 

\noindent T=binomial(m+1,2)*x\textasciicircum 2*(A1)\textasciicircum 2/A 

\noindent R=(binomial(m+1,2)*x*A1*T+binomial(m+1,3)*x\textasciicircum 2*(A1)\textasciicircum 2)/A

\noindent G=((m+1)*x*R*A\textasciicircum m+(m\textasciicircum 2+m)*x\textasciicircum 2*A1*R*A\textasciicircum (m-1)+binomial(m+2,2)*x*A\textasciicircum m*T

\noindent +binomial(m+1,2)*x\textasciicircum 2*A1*A\textasciicircum (m-1)*T+binomial(m+2,3)*x\textasciicircum 2*A1*A\textasciicircum m

\noindent +binomial(m+2,4)*x\textasciicircum 3*(A1)\textasciicircum 2*A\textasciicircum (m-1)-binomial(m+1,2)*x*A\textasciicircum (m-1)*T\textasciicircum 2)/(1-(m+1)*x*A\textasciicircum m)

\noindent d=m*(m+1)*x\textasciicircum 3*A3+m*(2*m+4)*x\textasciicircum 2*A2-24*G

\noindent d.full\textunderscore simplify()
\end{proof}

\section{Acknowledgements}
This research was conducted under the supervision of Joe Gallian at the University of Minnesota Duluth REU, funded by NSF Grant 1358659 and NSA Grant H98230-13-1-0273. The author heartily thanks Fabrizio Zanello and Joe Gallian for suggesting the topic of this project and for their valuable advice. The author also thanks Timothy Chow for his insightful discussions, Noah Arbesfeld for his comments, and the referees for their suggestions.

\end{document}